\documentclass[12pt,headings=big]{scrartcl}
 \usepackage{cmap}
 \usepackage{amsmath,amsxtra,amscd,amssymb,latexsym,stmaryrd,amsthm,lipsum}
 \usepackage{mathrsfs, calc, xcolor, array, graphicx, subcaption}

\allowdisplaybreaks

\usepackage[colorlinks=true,linkcolor=red!70!black,citecolor=green!70!black,
    urlcolor=magenta!70!black,backref]{hyperref}

  \usepackage{lmodern} 
  \usepackage[utf8]{inputenc}
  \usepackage[T1]{fontenc}

\addtokomafont{title}{\rmfamily\itshape}

\theoremstyle{plain}
\newtheorem{theo}{Theorem}[section]
\newtheorem*{theo*}{Theorem}
\newtheorem{coro}[theo]{Corollary}
\newtheorem{prop}[theo]{Proposition}
\newtheorem{lemm}[theo]{Lemma}
\newtheorem{theomain}{Theorem}

\newtheorem{hypo}{Assumption}

\theoremstyle{definition}
\newtheorem{defi}[theo]{Definition}

\newcommand*{\dd}%
  {\relax\ifnum\lastnodetype>0\mskip\medmuskip\fi\mathrm{d}}
\newcommand{\fspace}[1]{\mathscr{#1}}
\newcommand{\fX}{\fspace{X}}

\newcommand{\op}[1]{\mathrm{#1}}
\newcommand{\one}{\boldsymbol{1}}

\newcommand{\mchain}{\mathsf}

\newcommand{\pr}{\operatorname{\mathbb{P}}}

\newcommand{\Lip}{\operatorname{Lip}}
\newcommand{\BV}{\operatorname{BV}}
\newcommand{\var}{\operatorname{var}}

\newcommand{\diam}{\operatorname{\mathrm{diam}}}

\newlength{\hypobox}
\newlength{\gapbox}
\newcounter{hypop}
\renewcommand{\thehypop}{\textup{(H\arabic{hypop}')}}

\newcounter{gap}
\renewcommand{\thegap}{\textup{(SG\arabic{gap})}}

\title{Toy examples for effective concentration bounds}

\author{Beno\^{\i}t R. Kloeckner \thanks{Universit\'e Paris-Est, Laboratoire d'Analyse et de Mat\'ematiques Appliqu\'ees (UMR 8050), UPEM, UPEC, CNRS, F-94010, Cr\'eteil, France}}

\begin{document}

\maketitle

\begin{abstract}
In this note we prove a spectral gap for various Markov chains on various functional spaces. While proving that a spectral gap exists is relatively common, explicit estimates seems somewhat rare.

These estimates are then used to apply the concentration inequalities of \cite{K:concentration} (most of the present material was part of Section 3 of that article, which has been reduced to its core in the published version).
\end{abstract}

Let us recall briefly the notation and concentration inequalities from \cite{K:concentration}.

Let $(X_k)_{k\ge 0}$ be a Markov chain taking value in a general state space $\Omega$ with a unique stationary measure $\mu_0$, and let $\varphi:\Omega \to \mathbb{R}$ be a function (the ``observable''). We are interested in the speed of the convergence of the empirical average 
$\hat\mu_n(\varphi) := \frac1n \sum_{k=1}^n \varphi(X_k)$
to $\mu_0(\varphi)$. We denote by $\mu$ the law of $X_0$, which can be arbitrary. 

\begin{hypo}\label{hypo:X} The observable $\varphi$ belongs to a function space $\fX$ satisfying
\begin{enumerate}
\item its norm  $\lVert\cdot \rVert$ dominates the uniform norm: $\lVert\cdot \rVert\ge \lVert\cdot \rVert_\infty$,
\item $\fspace{X}$ is a Banach algebra, i.e. for all $f,g\in \fspace{X}$ we have $\lVert fg\rVert \le \lVert f\rVert \lVert g \rVert$,
\item $\fspace{X}$ contains the constant functions and
$\lVert\one\rVert= 1$ (where $\one$ denotes the constant function with value $1$).
\end{enumerate}
\end{hypo}

To the transition kernel $\mchain{M}$ is associated an averaging operator acting on $\fspace{X}$:
\[\op{L}_0 f(x) = \int_\Omega f(y) \dd m_{x}(y).\]
Since each $m_x$ is a probability measure, $\op{L}_0$ has $1$ as eigenvalue, with eigenfunction $\one$. 

\begin{hypo}\label{hypo:L}
The Markov chain $\mchain{M}$ satisfies the following:
\begin{enumerate}
\item $\op{L}_0$ acts as a bounded operator from $\fspace{X}$ to itself, and its operator norm $\lVert \op{L}_0\rVert$ is equal to $1$.
\item $\op{L}_0$ is contracting with gap $\delta_0>0$, i.e. there is a closed hyperplane $G_0 \subset \fspace{X}$ such that
\[ \lVert \op{L}_0 f \rVert \le (1-\delta_0) \lVert f\rVert \qquad \forall f\in G_0.\]
\end{enumerate}
\end{hypo}
The second hypothesis is a particular case of a \emph{spectral gap}: it implies in particular that $1$ is a simple isolated eigenvalue.

In \cite{K:concentration} the following two results where proved (plus a Berry-Esseen bound that we will not use here).
\begin{theomain}\label{theo:main-conc}
Assuming assumptions \ref{hypo:X} and \ref{hypo:L}, for all $n\ge 1+\frac{\log 100}{-\log(1-\delta_0/13)}$ it holds:
\[ \pr_\mu\Big[\lvert\hat\mu_n(\varphi)-\mu_0(\varphi)\rvert\ge a\Big]    \le \begin{cases}\displaystyle
2.488 \exp\Big(-n  \frac{\delta_0}{13.44\delta_0+8.324}  \frac{a^2}{\lVert\varphi\rVert^2}\Big)
 & \displaystyle \mbox{if }\frac{a}{\lVert \varphi\rVert} \le \frac{\delta_0}{3} \\[6\jot]
 \displaystyle
2.624 \exp\Big( -n\frac{0.98 \delta_0^2}{12+13\delta_0} \Big(\frac{a}{\lVert\varphi\rVert}-0.254\delta_0\Big) \Big)
 &\mbox{otherwise.}
\end{cases}\]
\end{theomain}
(We will often use the strenghtened hypothesis $n\ge 60/\delta_0$ for simplicity.)

\begin{theomain}\label{theo:main-second}
Assuming assumptions \ref{hypo:X} and \ref{hypo:L}, for all $n\ge  \frac{60}{\delta_0}$, all $U \ge \sigma^2(\varphi)$ and all
$a \le \frac{U}{\lVert\varphi\rVert} \log\Big(1+\frac{\delta_0^2}{12+13\delta_0}\Big)$ it holds:
\[\pr_\mu\big[\lvert \hat\mu_n(\varphi)-\mu_0(\varphi)\rvert\ge a\big]
  \le 2.637 \exp\left(-n\cdot\Big(\frac{a^2}{2U} - 10(1+\delta_0^{-1})^2 \frac{\lVert\varphi\rVert^3a^3}{U^3}\Big)\right).\]
\end{theomain}
Above, we use the notation $\sigma^2(\varphi)=\mu_0(\varphi^2)-(\mu_0\varphi)^2+2\sum_{k\ge 1} \mu_0(\varphi \op{L}_0^k \bar\varphi)$
where $\bar\varphi = \varphi-\mu_0(\varphi)$. This ``dynamical variance'' is precisely the variance appearing in the CLT. 

While it is well known that the presence of a spectral gap ensures classical limit theorems, these results turn explicit contraction estimates into explicit non-asymptotic results. The main goal of this note is to compute lower bounds on $\delta_0$ for several pairs of Markov chains and functional spaces. We shall apply the above result for illustration, and compare to previous results when available.

\section{Preliminary lemma}

In each example below we will use the following lemma which, in the spirit of Doeblin-Fortet and Lasota-Yorke inequalities, enables to turn an exponential contraction in the  ``regularity part'' of a functional norm into a spectral gap.
\begin{lemm}\label{lemm:gap}
Consider a normed space $\fspace{X}$ of (Borel measurable, bounded) functions $\Omega\to\mathbb{R}$, with norm $\lVert \cdot\rVert = \lVert \cdot \rVert_\infty+V(\cdot)$ where $V$ is a semi-norm (usually quantifying some regularity of the argument, such as $\Lip$ or $\BV$).

Assume that for some constant $C>0$, for all probability $\mu$ on $\Omega$ and for all $f\in\fspace{X}$ such that $\mu(f)=0$, 
$\lVert f\rVert_\infty \le C V(f)$.

Let $\op{L}_0\in\fspace{B}(\fspace{X})$ and assume that for some $\theta\in(0,1)$ and all $f\in\fX$: 
\[\rVert \op{L}_0f\rVert_\infty \le \lVert f\rVert_\infty \quad\mbox{and}\quad V(\op{L}_0 f)\le \theta V(f)\]
and having eigenvalue $1$ with an eigenprobability $\mu_0$, i.e. $\op{L}_0^*\mu_0=\mu_0$.

Then $\op{L}_0$ is contracting with gap at least 
\[\delta_0 = \frac{1-\theta}{1+C\theta}.\]
\end{lemm}

The condition $\lVert f\rVert_\infty \le C V(f)$ is often valid in practice (assuming $\Omega$ has finite diameter for spaces such as $\Lip(\Omega)$): the condition that $\mu(f)=0$ implies that $f$ vanishes (if functions in $\fX$ are continuous) or at least takes both non-positive and non-negative values, and $V(f)$ usually bounds the variations of $f$, implying a bound on its uniform norm.

\begin{proof}
Let $f\in \ker\mu_0$; then $\lVert \op{L}_0f\rVert_\infty\le \lVert f\rVert_\infty$ and $\op{L}_0f\in\ker\mu_0$, so that $\lVert \op{L}_0 f\rVert_\infty\le CV(\op{L}_0f)\le C\theta V(f)$.

Denote by $t\in[0,1]$ the number such that $\lVert f\rVert_\infty = t\lVert f\rVert$ (and therefore $V(f)=(1-t)\lVert f\rVert$).
The above two controls on $\lVert \op{L}_0(f)\rVert_\infty$ can then be written as $\lVert \op{L}_0(f)\rVert_\infty \le \min\big(t, C\theta(1-t)\big)\lVert f\rVert$ and using $V(\op{L}_0f)\le \theta V(f)$ again we get
\begin{align*}
\lVert \op{L}_0(f)\rVert &\le \min\big(t+\theta(1-t), (C+1)\theta(1-t)\big)\lVert f\rVert \\
\lVert (\op{L}_0)_{|\ker\mu_0}\rVert &\le \max_{t\in[0,1]} \min\big(t+\theta(1-t), (C+1)\theta(1-t)\big).
\end{align*}
The maximum is reached when $t+\theta(1-t) = (C+1)\theta(1-t)$, i.e. when $t=C\theta/(1+C\theta)$, at which point the value in the minimum is 
$(C+1)\theta/(C\theta+1) \in(0,1)$. We get contraction with gap $1- (C+1)\theta/(C\theta+1)$, as claimed.
\end{proof}

\section{Chains with Doeblin's minorization}\label{sec:Doeblin}

We start with a warm-up in the simplest example of a Banach Algebra of functions, the space of measurable bounded functions $L^\infty(\Omega)$.\footnote{We do not have a single reference measure here, which is why we consider genuinely bounded functions rather than essentially bounded functions.} To fit our framework, we will need to endow $L^\infty(\Omega)$ with the norm $\lVert\cdot\rVert_S = \lVert f\rVert_\infty + S(f)$ where
\[S(f) :=  \sup_{x,y\in\Omega} \lvert f(x)-f(y)\rvert=\sup f-\inf f\]
measures how ``spread out'' $f$ is, which we need to manage separately from the magnitude of $f$.
Of course, this norm is equivalent to the uniform norm, and it is easily checked what we still get a Banach Algebra.

Observe that convergence of measures in duality to $L^\infty(\Omega)$ is convergence in total variation, and the most usual normalization is
\[d_{\mathrm{TV}}(\mu,\nu) := \sup_{S(f)=1} \big\lvert \mu(f) - \nu(f)\big\rvert.\]
For a transition kernel $\mchain{M}$, having an averaging operator $\op{L}_0$ with a spectral gap is a very strong condition, called \emph{uniform ergodicity}.

Glynn and Ormoneit \cite{Glynn2002} and Kontoyiannis, Lastras-Montaño and Meyn \cite{Kontoyiannis-LMM} gave explicit concentration results for such chains, using the characterization of uniform ergodicity by the \emph{Doeblin minorization condition}: there exist an integer $\ell\ge 1$, a positive number $\beta$ and a probability measure $\omega$ on $\Omega$ such that for all $x\in\Omega$ and all Borel set $B\subset \Omega$:
\begin{equation}
m_x^\ell(B) \ge \beta\omega(B)
\label{eq:Doeblin}
\end{equation}
where $m_x^\ell$ is the law of $X_\ell$ conditionally to $X_0=x$.

We shall look at the case $\ell=1$, which fits better in our context. For arbitrary value of $\ell$, one can in practice apply the result to each extracted chain $(X_{k_0+k\ell})_{k\ge 0}$.
\begin{prop}\label{lemm:Doeblin}
If $\mchain{M}$ satisfies Doeblin's minorization condition \eqref{eq:Doeblin} with $\ell=1$, then its averaging operator $\op{L}_0$ is contracting on $L^\infty(\Omega)$ with gap $\beta/(2-\beta)$.
\end{prop}

\begin{proof}
This is simply the classical maximal coupling method in a functional guise. For each $x\in\Omega$  decompose $m_x$ into $\beta\omega$ and $r_x := m_x-\beta\omega$ (which is a positive measure of mass $1-\beta$). Recall that we denote by $\mu_0$ the stationary measure of $\mchain{M}$. For all $f\in L^\infty(\Omega)$ we have:
\begin{align*}
\op{L}_0 f(x) &= \beta\omega(f)+ r_x(f) \\
\op{L}_0 f(x) - \op{L}_0 f(y) &= \int (r_x(f)-r_y(f)) \dd\mu_0(y) \\
\big\lvert \op{L}_0 f(x) - \op{L}_0 f(y) \big\rvert &\le \int (1-\beta)S(f) \dd\mu_0(y) \\
S(\op{L}_0 f)   &\le (1-\beta)S(f).
\end{align*}
We can thus apply Lemma \ref{lemm:gap} with $C=1$ and $\theta=1-\beta$, obtaining a spectral gap of size $\beta/(2-\beta)$.
\end{proof}

\begin{coro}\label{theo:Doeblin}
If $\mchain{M}$ satisfies Doeblin's minorization condition \eqref{eq:Doeblin} with $\ell=1$ and $\varphi : \Omega\to[-1,1]$, for all $n\ge 120/\beta$ and all $a\le \beta/2$ it holds
\begin{equation*}
\pr_\mu\Big[\lvert\hat\mu_n(\varphi)-\mu_0(\varphi)\rvert\ge a\Big]    \le
 2.5 \exp\big(- n a^2 \cdot \frac{\beta}{150 + 47\beta}\big).
\end{equation*}
\end{coro}

\begin{proof}
We have here $\lVert \varphi\rVert_S\le 3$ and, by Lemma \ref{lemm:Doeblin}, $\delta_0\ge \beta/(2-\beta)\ge \beta/2$. It then suffices to apply Theorem \ref{theo:main-conc} and round constants up.
\end{proof}

The exponent is proportional to $\beta$, which is the correct rate and improves on \cite{Glynn2002} and \cite{Kontoyiannis-LMM} which get a $\beta^2$; but
Paulin obtains better constants in this case \cite{Paulin2015} (Corollary 2.10), and we do not study this example further.

\section{Discrete hypercube}\label{sec:cube}

Let us consider the same toy example as Joulin and Ollivier \cite{JO}, the lazy random walk (aka Gibbs sampler, aka Glauber dynamics) on the discrete hypercube $\{0,1\}^N$: the transition kernel $\mchain{M}$ chooses randomly uniformly a slot $i\in\{1,\dots, N\}$ and replaces it with the result of a fair coin toss, i.e. 
\[m_x = \frac12 \delta_x + \sum_{y\sim x} \frac{1}{2N} \delta_y.\]

We consider two kind of observables: the ``polarization''$\rho:\{0,1\}^N \to \mathbb{R}$ giving the proportion of $1$'s in its argument, and the characteristic function $\one_S$ of a subset $S\subset \{0,1\}^N$. In this second example, we will in particular consider the simple case $S=[0]:=\{(0,x_2,\dots,x_N) \colon x_i\in\{0,1\}\}$.

%
%

\subsection{Spectral gap estimates}

The discrete hypercube $\{0,1\}^N$ is endowed with the Hamming metric: if $x=(x_1,\dots,x_N)$ and $y=(y_1,\dots,y_N)$, then $d(x,y)$ is the number of indexes $i$ such that $x_i\neq y_i$. Two elements at distance $1$ are said to be adjacent, denoted by $x\sim y$. 

We denote by $E$ the set of tuples $\epsilon=(\epsilon_i)_{1\le i\le N}$ such that exactly one of the $\epsilon_i$ is $1$. Identifying $\{0,1\}$ with $\mathbb{Z}/2\mathbb{Z}$, an edge thus writes $(x,x+\epsilon)$ for some $x\in\{0,1\}^N$ and some $\epsilon\in E$.

We shall consider several function spaces to showcase the flexibility of the spectral method; since the space $\{0,1\}^N$ is finite, we always consider the space of all functions $\{0,1\}^N\to\mathbb{R}$, and it is the considered norm which will matter. Let us define:
\begin{itemize}
\item $\lVert f \rVert_{L} = \lVert f \rVert_\infty + \Lip(f)$: this is the standard Lipschitz norm;
\item $\lVert f \rVert_{dL} = \lVert f \rVert_\infty + N\Lip(f)$: this is the Lipschitz norm with a weigth to the regularity part equal to the diameter;
\item $\lVert f \rVert_{W} = \lVert f \rVert_\infty + W(f)$ where 
\[W(f) = \sup_{x\in\{0,1\}^N} \sum_{\epsilon\in E} \lvert f(x+\epsilon)-f(x)\rvert;\]
this norm stays small for functions having large variations only in few directions (small ``local total variation'').
\end{itemize}

We shall use later the following non-trivial comparison with $\lVert\cdot\rVert_S$.
\begin{lemm}[Fedor Petrov \cite{Petrov}]\label{lemm:Petrov}
For all $f:\{0,1\}^N\to\mathbb{R}$ we have 
\[\max f - \min f\le W(f).\]
\end{lemm}

\begin{proof}
Without lost of generality, we can assume $W(f)\le 1$ and $f(0,0,\dots, 0)=0$, and reduce to proving $f(1,1,\dots,1)\le 1$.

Define the \emph{cost} of a path $x^0,x^2,\dots,x^k$ as the number $\sum_{i=0}^{k-1} \lvert f(x^{i+1})-f(x^i) \rvert$, and let $\Sigma$ be the sum of the costs of all paths of length $N$ from $(0,0,\dots,0)$ to $(1,1,\dots,1)$.
We shall prove that $\Sigma\le N!$, and since there are $N!$ such paths one of them will have cost at most $1$, proving the lemma.

We call ``level'' of $x\in\{0,1\}^N$ the number of $1$s among the coordinates of $x$, and denote it by $\lvert x\rvert$.
For each $i\in\{0,1,\dots,N-1\}$, define $p_i= \frac{i!(N-i)!}{N+1}$. Then all $p_i$ are positive and $p_i+p_{i+1}=i!(N-i-1)!$ is precisely the number of paths that use any given edge from level $i$ to level $i+1$.

The contribution to $\Sigma$ of an edge $(x,x+\epsilon)$ from level $i$ to level $i+1$ is thus $i!(N-i-1)!\lvert f(x+\epsilon)-f(x) \rvert$, which we split into two parts, one $p_i \lvert f(x+\epsilon)-f(x) \rvert$ attributed to $x$ and the other $p_{i+1}\lvert f(x+\epsilon)-f(x) \rvert$ to $x+\epsilon$. It follows 
\[\Sigma \le \sum_{x\in\{0,1\}^N} p_{\lvert x\rvert} W(f) \le \sum_{i=0}^N p_i \binom{N}{i} = \sum_{i=0}^{N-1}(p_i+p_{i+1})\binom{N-1}i=N(N-1)!=N!\]
as desired. 
\end{proof}

We get the following gap estimates.
\begin{theo}\label{theo:Hamming}
Each of the norm $\lVert\cdot\rVert_L$, $\lVert\cdot\rVert_{dL}$ and $\lVert\cdot\rVert_{W}$ turns the space of all functions $\{0,1\}^N\to\mathbb{R}$ into a Banach algebra where $\one$ has norm $1$.

Moreover the averaging operator $\op{L}_0$ of the transition kernel $\mchain{M}$ has operator norm $1$, and is contracting with gap respectively $1/N^2$, $1/(2N-1)$ and $1/(4N-1)$ in the norms $\lVert\cdot\rVert_L$, $\lVert\cdot\rVert_{dL}$ and $\lVert\cdot\rVert_{W}$.
\end{theo}

\begin{proof}
Each norm considered here has the form $\lVert \cdot \rVert = \lVert \cdot \rVert_\infty + V(\cdot)$ for some semi-norm $V$ such that 
$V(fg) \le \lVert f\rVert_\infty V(g) + V(f)\lVert f\rVert_\infty$; it follows that the considered spaces are Banach algebras. All the other properties but the contraction are trivial.

To prove the contraction, we simply apply Lemma \ref{lemm:gap}. 
First, it is well-known that  for all $\varphi:\{0,1\}^N \to \mathbb{R}$,
\[\Lip(\op{L}_0\varphi) \le (1-1/N) \Lip(\varphi)\] 
(in the parlance of \cite{Ollivier}, $\mchain{M}$ is positively curved with $\kappa=1/N$).

In the case of $\lVert\cdot\rVert_L$, we get $\theta=1-1/N$ and $C=N$ (since a function of vanishing average must take positive and negative values, and $\diam \{0,1\}^N=N$), hence a contraction with gap $1/N^2$. In the case of $\lVert\cdot\rVert_{dL}$, the normalizing factor gives $C=1$ (and we still have $\theta=1-1/N$), hence a spectral gap of size $1/(2N-1)$.

To deal with $\lVert\cdot \rVert_W$, we first show that in Lemma \ref{lemm:gap} we can take  $\theta=1-1/(2N)$.
\begin{align*}
W(\op{L}_0\varphi) &= \sup_x \sum_{\epsilon\in E} \bigg\lvert \frac{1}{2}\varphi(x+\epsilon)+\frac{1}{2N}\sum_{\eta\in E} \varphi(x+\eta+\epsilon)-\frac{1}{2}\varphi(x)-\frac{1}{2N}\sum_{\eta\in E} \varphi(x+\eta) \bigg\rvert\\
  &= \sup_x \sum_{\epsilon\in E} \bigg\lvert \Big(\frac{1}{2}-\frac{1}{2N}\Big) \varphi(x+\epsilon) +\frac{1}{2N}\sum_{\eta\neq \epsilon} \varphi(x+\eta+\epsilon) \\
  &\qquad\qquad\qquad\qquad - \Big(\frac{1}{2}-\frac{1}{2N}\Big) \varphi(x)- \frac{1}{2N}\sum_{\eta\neq \epsilon} \varphi(x+\eta) \bigg\rvert \\
  &\le \sup_x \frac{N-1}{2N} \sum_{\epsilon\in E} \lvert \varphi(x+\epsilon)-\varphi(x) \rvert + \frac{1}{2N}\sum_{\epsilon\in E}\sum_{\eta\neq \epsilon} \lvert \varphi(x+\epsilon+\eta)-\varphi(x+\eta) \rvert \\
  &\le \frac{N-1}{2N} W(\varphi)+ \frac{1}{2N}\sup_x \sum_{y\sim x} \sum_{\epsilon\in E} \lvert \varphi(y+\epsilon)-\varphi(y) \rvert.
\end{align*}
Hence we obtain $W(\op{L}_0\varphi)  \le \big(1-\frac{1}{2N}\big)W(\varphi)$.

Then Lemma \ref{lemm:Petrov} shows that we can take $C=1$, providing a spectral gap of size $1/(4N-1)$.
\end{proof}

\subsection{Concentration inequalities}

Let us combine \ref{theo:Hamming} with \ref{theo:main-conc} and \ref{theo:main-second} to obtain explicit concentration estimates. We will not compute the explicit constants, and concentrate on the dependency with the parameters $a$ and $N$.

Consider first the ``polarization'' observable $\rho:\{0,1\}^N \to \mathbb{R}$, where $\rho(x)$ is the proportion of $1$'s in the word $x$. We have
\[\lVert \rho \rVert_L = 1+\frac1N, \qquad \lVert \rho \rVert_{dL} = 2, \qquad \lVert \rho \rVert_W = 2.\]
To use Theorem \ref{theo:main-conc} with optimal efficiency, assuming $a$ will be small enough, we need to maximize $\delta_0/\lVert \rho\rVert^2$. Here, we shall thus use the norm $\lVert\cdot\rVert_{dL}$. For $a\lesssim N$, Theorem \ref{theo:main-conc} shows that we need at most $O(N/a^2)$ iterations to have a good convergence to the actual mean; meanwhile Joulin and Ollivier only need $O(1/a^2)$, but for concentration around the expectancy of the empiric process, not around the expectancy with respect to the stationary measure. Without burn-in, one also needs to bound the bias, which approaches zero in time $O(N/a)$ according to the bound of Joulin and Ollivier, for a total run time of $O(N/a+1/a^2)$. 
With burn-in, they need a run time of $O(N+1/a^2)$.

For $1/N \lesssim a\lesssim 1$, we enter our exponential regime while staying inside Joulin-Ollivier's Gaussian window; Theorem \ref{theo:main-conc} shows we need no more than $O(N^2/a)$ iterations, while \cite{JO} still gives a bound of $O(N + 1/a^2)$.

In this example, Joulin and Ollivier get a sharper result; this seems to be explained in one part by the fact that we do not get to decouple the bias from the convergence of expectancies, and in another part by our need to have a Banach algebra, hence to include the uniform norm in our norm.

Consider now the potential $\one_S$, the indicator function for a (non-trivial) set $S$. This function is only $1$-Lipschitz, so that we have $\lVert \one_S \rVert_L=2$ and $\lVert \one_S\rVert_{dL} = 1+N$. If we insist on using a Lipschitz norm, the unormalized one is thus better and with $\delta_0=1/N^2$ Theorem \ref{theo:main-conc} shows that we need (in the Gaussian regime) $O(N^2/a^2)$ iterations to ensure the error is probably less than $a$, which is the same order of magnitude than given by \cite{JO} with a worse constant, $\sim 34$ instead of $8$.
But here we have two ways to improve on this bound.

The first one is to use Theorem \ref{theo:main-second}.
When $S=[0]:=\{0x_2x_3\cdots x_N \in\{0,1\}^N \}$,
the dynamical variance can be computed explicitly
(distinguish the cases when the first digit has been changed an odd or even number of times, and observe that at each step the probability of changing the first digit is $1/2N$):
\[\mu_0(\one_{[0]}^2)-(\mu_0\one_{[0]})^2=1/4 \quad\mbox{and}\quad \sum_{k\ge1} \mu_0(\one_{[0]} \op{L}_0^k \bar\one_{[0]}) = \frac14 \sum_{k\ge 1} \Big(\frac{N-1}{N}\Big)^k = \frac{N-1}{4}.\]
This gives $\sigma^2(\one_S)\simeq N/2$. Switching back to the norm $\lVert\cdot\rVert_{dL}$, when  $a\lesssim 1/N^2$ and $n\ge 60N^2$, in Theorem \ref{theo:main-second} the positive term in the exponential is negligible compared to the main term which is $-na^2/N$. In particular $O(N/a^2)$ iterations suffice to get a small probability for a deviation at least $a$: compared to Joulin and Ollivier, we gain one power of $N$ in this regime (and the optimal constant $1$ in the leading term of the rate) but only for very small values of $a$.\footnote{If we want to consider $a$ of the order of $1/N$, we can then take $U\simeq N^2$ to enlarge the window, at the cost of a weaker leading term. We get a bound similar to the one of Joulin-Ollivier, possibly with a smaller constant (depending on the value of $a$).}
This choice of $S$ might seem very specific, but for less regular $S$ the gain should be greater for sufficiently smaller $a$. For example, if $S$ contains half the vertices and every vertex $x\in\{0,1\}^N$ has exactly $2Np$ neighbors with the same $\one_S$ value, the above computation of variance gives $\sigma^2(\one_S)=\frac14+\frac{1-2p}{4p}$. We shall call a family of sets $S_N\in\{0,1\}^N$ ``scrambled'' when the indicator functions $\one_{S_N}$ have bounded variance (independently of $N$) with respect to the lazy random walk; by abuse, we shall speak of a scrambled set for a member of such a family. For scrambled sets taking $n=O(1/a^2)$ is sufficient: there is no dependency on the dimension. A further study of scrambled sets seems an interesting direction of work.

The second way to improve our first estimate is to use the norm $\lVert\cdot\rVert_W$ in Theorem \ref{theo:main-conc}. Then $\lVert \one_{[0]} \rVert_W=2$ and $\delta_0 \simeq 1/N$. For $a\lesssim 1/N$, Theorem \ref{theo:main-conc} ensures that we need only $O(N/a^2)$ iterations to have a good convergence to the actual mean, which is again the optimal order of magnitude (since it corresponds to the CLT) but obtained on a much larger window than with Theorem \ref{theo:main-second}. This extends to all observables with $W(\varphi)\lesssim 1$; observe that this domain of applicability is quite complementary to the domain of applicability of the previous paragraph.

\section{Bernoulli convolutions and observables of bounded variation}

We now consider the ``Bernoulli convolution'' of parameter $\lambda\in(0,1)$, defined as the law $\beta_\lambda$ of the random variable 
\[\sum_{k\ge 1} \epsilon_k \lambda^k\]
where the $\epsilon_k$ are independent variables taking the value $1$ with probability $1/2$ and the value $-1$ with probability $1/2$ (see \cite{K:concentration} for a brief account, and \cite{Peres2000} for more information on these measures, which are the object of intense scrutiny for decades).

\begin{figure}
\centering
\begin{subfigure}[t]{.333\linewidth}
  \centering
  \includegraphics[width=\linewidth]{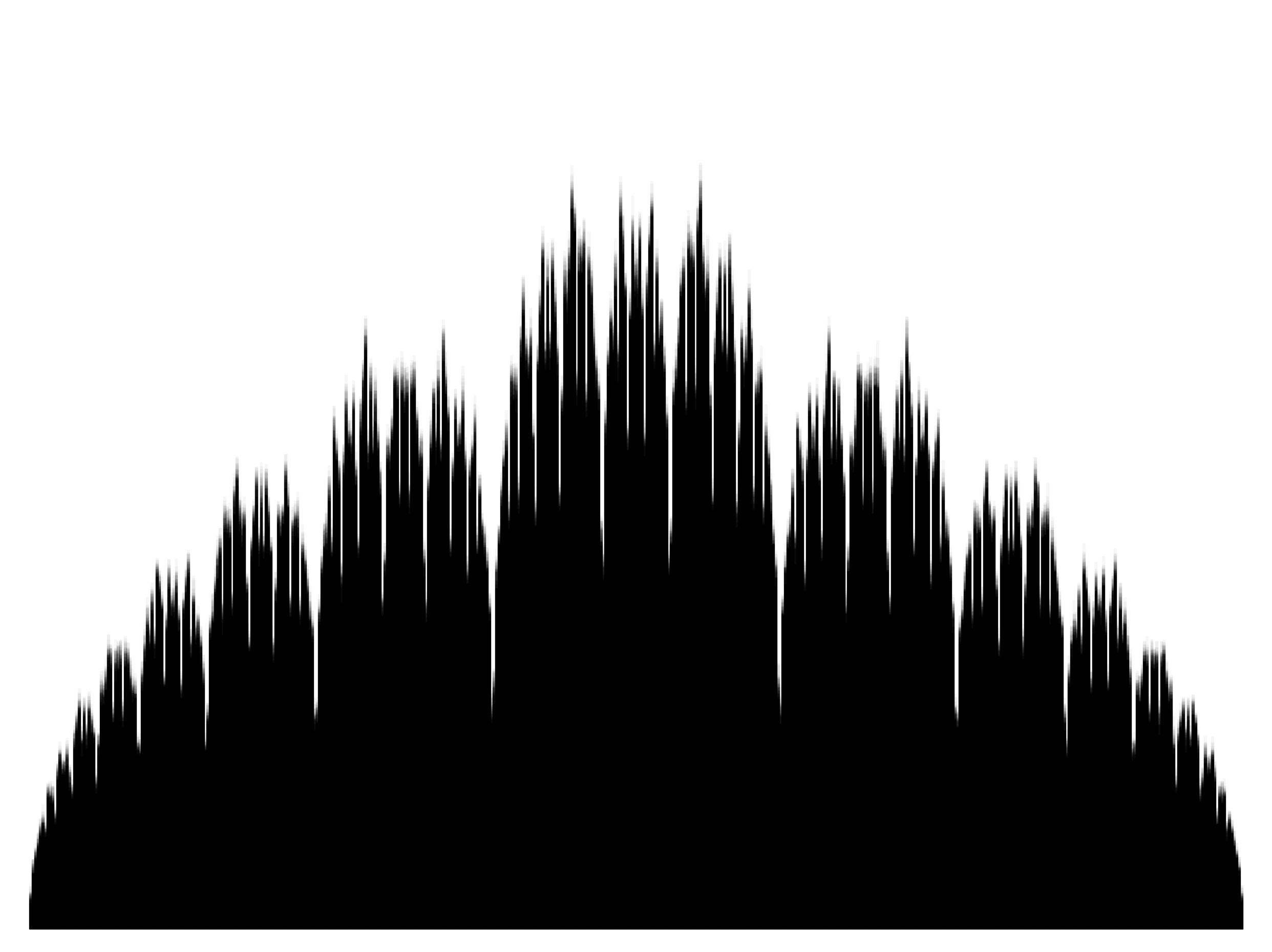} 
  \caption{$\lambda=\frac{\sqrt{5}-1}{2}\simeq 0.618$}
\end{subfigure}%
~
\begin{subfigure}[t]{.333\linewidth}
  \centering
  \includegraphics[width=\linewidth]{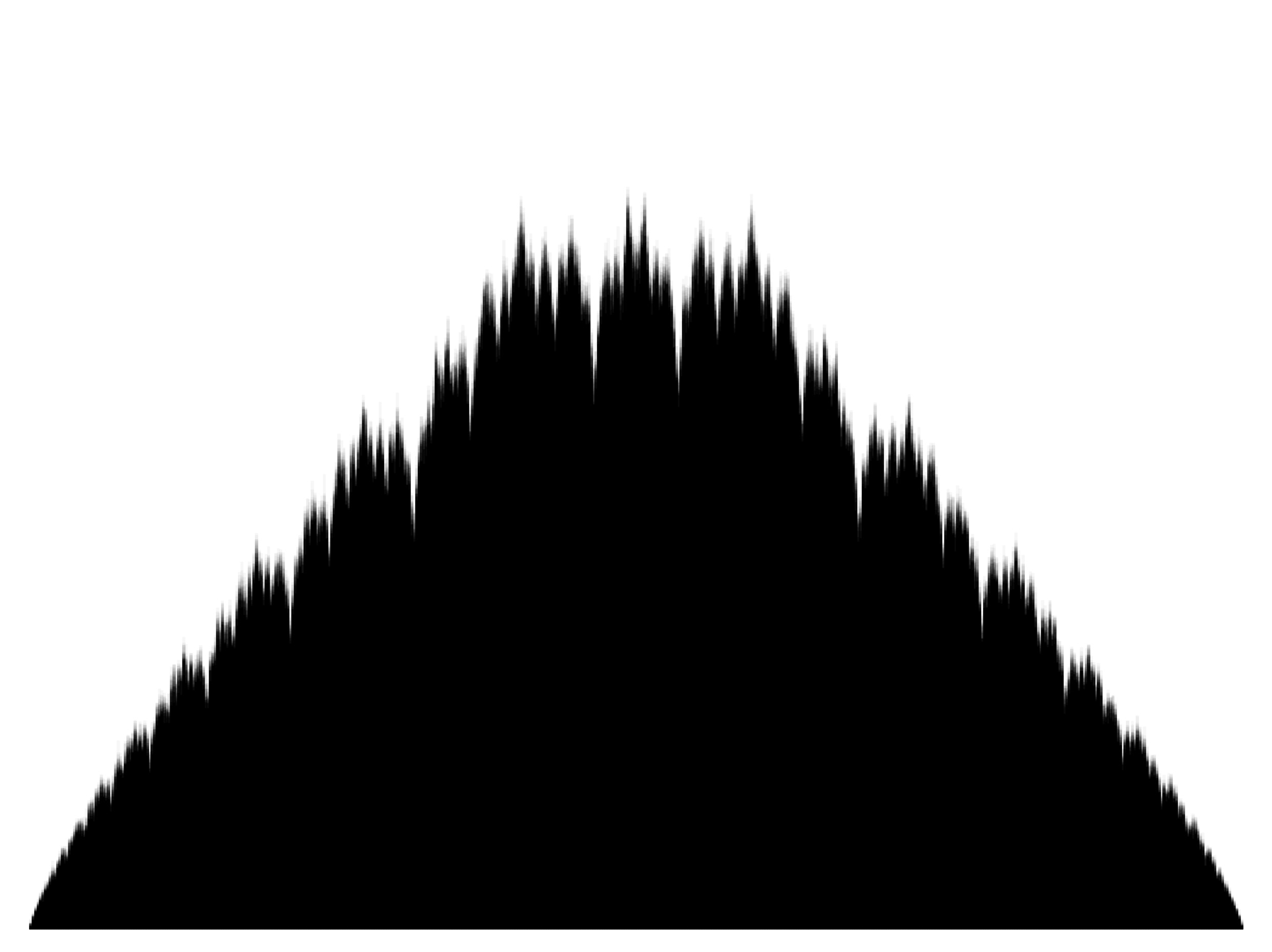}
  \caption{$\lambda=\frac{e}{4}\simeq 0.680$}
\end{subfigure}%
~
\begin{subfigure}[t]{.333\linewidth}
  \centering
  \includegraphics[width=\linewidth]{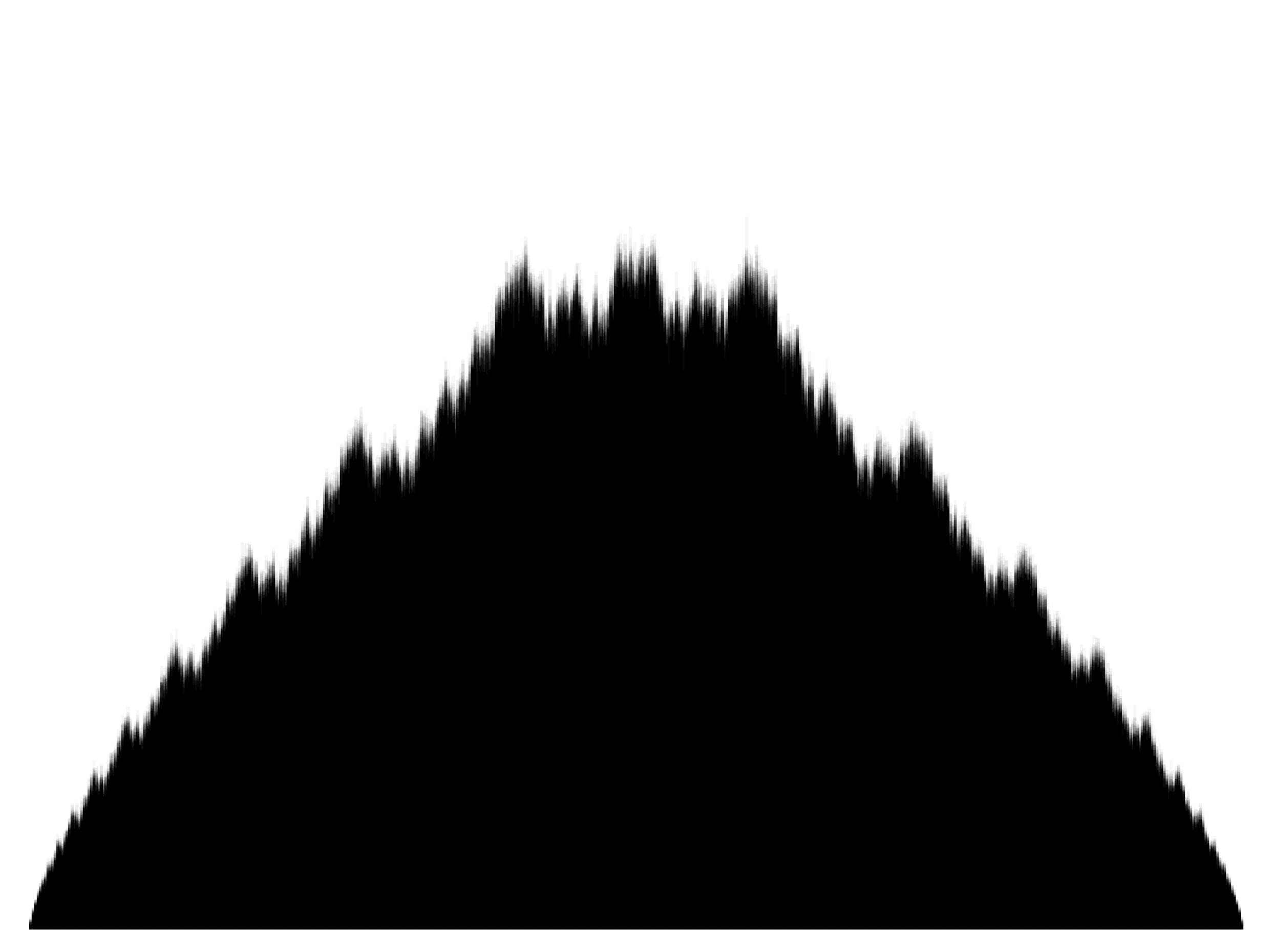}
  \caption{$\lambda=\frac23 \simeq 0.667$}
\end{subfigure}
\caption{Histogram of the empirical distribution of the Markov chain associated to $(T_0,T_1)$, with $X_0=0$, binned in $500$ subintervals (averaged image over $30$ independent runs of $10^6$ points each). Parameter $\lambda$ is the inverse of a Pisot number on the left, a very well approximable irrational at the center, rational on the right.}\label{fig:Bernoulli}
\end{figure}


One can realize naturally $\beta_\lambda$ as the stationary law of the Markov transition kernel $\mchain{M}=(m_x)_{x\in \mathbb{R}}$ defined by
\[m_x = \frac12 \delta_{T_0(x)}+ \frac12\delta_{T_1(x)}\]
where $T_0(x) = \lambda x-\lambda$ and $T_1(x)=\lambda x+\lambda$ (this is a particular case of an Iterated Function System).

In order to evaluate $\beta_\lambda(\varphi)$ by a MCMC method, one cannot use the methods developed for ergodic Markov chains since, conditionally to $X_0=x$, the law $m_x^k$ of $X_k$ is atomic and thus singular with respect to $\beta_\lambda$: $d_{\mathrm{TV}}(m_x^k,\beta_\lambda) = 1$ for all $k$. The convergence only holds for observables satisfying some regularity assumption, and it is natural to ask what regularity is needed.

For a Lipschitz observable $\varphi$ one only need to observe that $\mchain{M}$ has positive curvature in the sense of Ollivier (this is easy using the coupling $\frac12 \delta_{(T_0(x),T_0(y))} + \frac12 \delta_{(T_1(x),T_1(y))}$ of $m_x$ and $m_y$) and apply \cite{JO}. But what if $\varphi$ is not Lipschitz (or has large Lipschitz constant)?  We shall consider observables of bounded variation, a regularity which has the great advantage over Lipschitz to include the characteristic functions of intervals.

\begin{defi}
Given an interval $I\subset \mathbb{R}$, we consider the Banach space $\BV(I)$ of \emph{bounded variation} functions $I\to \mathbb{R}$, defined by the norm $\lVert\cdot\rVert_{\BV}=\lVert\cdot \rVert_\infty + \var(\cdot,I)$ where
\[\var(f,I) := \sup_{x_0<x_1<\dots<x_p \in I} \sum_{j=1}^p \lvert f(x_j)-f(x_{j-1})\rvert\]
(the uniform norm is usually replaced by the $L^1$ norm, but when $I$ is bounded our choice is equivalent up to a constant, it does not single out the Lebesgue measure, and most importantly it ensures that $\BV(I)$ is a Banach algebra).
\end{defi}

Important features of total variation are:
\begin{itemize}
\item its extensiveness:  $\var(f,I) \ge \var(f,J)+\var(f,K)$ whenever $J,K$ are disjoint subintervals of $I$,
\item its invariance under monotonic maps: $\var(f\circ T,I)=\var(f,T(I))$ whenever $T$ is monotonic.
\end{itemize}

It turns out that the averaging operator $\op{L}_0$ of the transition kernel $\mchain{M}$ has a spectral gap for all $\lambda$, but is not a contraction when $\lambda>1/2$ (i.e. we have an inequality $\lVert \op{L}_0^n(f) \rVert \le C(1-\delta_0)^n \lVert f\rVert$ on a closed hyperplane, but with $C>1$). In yet other words, an iterate of $\op{L}_0$ is a contraction, and to apply directly Theorems \ref{theo:main-conc} and \ref{theo:main-second} we need to consider an extracted Markov chain $(X_{\ell k})_{k\ge 0}$ for some $\ell$.

Let $I_\lambda$ be the attractor of the IFS $(T_0,T_1)$, i.e. the interval whose endpoints  are the fixed points of $T_0$ and $T_1$:
\[I_\lambda=\big[\frac{-\lambda}{1-\lambda},\frac{\lambda}{1-\lambda}\big].\]
Given a word $\omega = \omega_1\omega_2\dots\omega_k$ in the letters $0$ and $1$, we define
\[T_\omega = T_{\omega_1} \circ T_{\omega_2} \circ \dots \circ T_{\omega_k} : I_\lambda\to I_\lambda.\]

\begin{theo}
If $\lambda^\ell<\frac12$, then $\op{L}_0^\ell$ has a spectral gap on $\BV(I_\lambda)$ of size $1/(2^{\ell+1}-1)$ and constant $1$.
\end{theo}

\begin{proof}
Let $I_\lambda^-, I_\lambda^+$ be the left and right halves of $I_\lambda$, i.e.
$I_\lambda^-=\big[\frac{-\lambda}{1-\lambda},0\big)$ and $I_\lambda^+ = \big(0,\frac{\lambda}{1-\lambda}\big]$.

Let $f\in\BV(I_\lambda)$ and observe that the condition $\lambda^\ell<\frac12$ ensures that $T_{00\dots 0}(I_\lambda)$ and $T_{11\dots 1}(I_\lambda)$ are disjoint  (they have length $<\frac12\lvert I_\lambda\rvert$ and each contains an endpoint of $I_\lambda$). Then:
\begin{align*}
\var(\op{L}_0^\ell f,I_\lambda) &\le \frac{1}{2^\ell} \sum_{\omega\in\{0,1\}^\ell} \var(f\circ T_\omega,I_\lambda) 
  \le \frac{1}{2^\ell} \sum_{\omega\in\{0,1\}^\ell} \var(f,T_\omega(I_\lambda))\\
  &\le \frac{1}{2^\ell} \Big( \var(f,T_{00\dots0}(I_\lambda))+ \var(f,T_{11\dots1}(I_\lambda) + \sum_{\substack{\omega \neq 00\dots0 \\ \phantom{\omega}\neq 11\dots1}} \var(f,I_\lambda) \Big)\\
  &\le \frac{1}{2^\ell}\big( \var(f,I_\lambda)+(2^\ell-2)\var(f,I_\lambda) \big) \\
\var(\op{L}_0^\ell f,I_\lambda)  &\le (1-2^{-\ell}) \var(f,I_\lambda).
\end{align*}
Applying Lemma \ref{lemm:gap} with $C=1$ and $\theta=1-2^{-\ell}$ yields the claim.
\end{proof}

This enables us to apply our result to estimate $\beta_\lambda(\varphi)$ for any $\varphi$ of bounded variation. For example, Theorem \ref{theo:main-conc} yields the following.
\begin{coro}
Let $\lambda\in (\frac12,1)$ and let $\ell$ be an integer such that $\lambda^\ell < \frac12$. Consider a Markov chain $(X_k)_{k\ge0}$ with transition probability $2^{-\ell}$ from $x\in I_\lambda$ to $T_\omega(x)$, for each $\omega\in\{0,1\}^\ell$. For any starting distribution $X_0 \sim \mu$, any $\varphi\in\BV(I_\lambda)$, 
any positive $a <\lVert \varphi\rVert_{\BV}/3(2^{\ell+1}-1)$ and any $n\ge 120\cdot 2^\ell$
 we have
\[ \pr_\mu\Big[\lvert\hat\mu_n(\varphi)-\mu_0(\varphi)\rvert\ge a\Big]    \le 2.488 \exp\Big(- \frac{n a^2}{\lVert\varphi\rVert_{\BV}^2 (16.65 \cdot 2^\ell + 5.12)} \Big).\]
\end{coro}
To the best of our knowledge, this example could not be handled effectively by previously known results. For example \cite{Gomez-Dartnell2012} needs the observable to be at least $C^2$ to have explicit estimates, and they do not give a concentration inequality.

\bibliographystyle{amsalpha}
\bibliography{concentration}
\end{document}